\documentclass[12pt, a4paper]{amsart}

\usepackage[hmargin=30mm, vmargin=25mm, includefoot, twoside]{geometry}
\usepackage[bookmarksopen=true]{hyperref}

\usepackage{amsfonts,amssymb,verbatim}
\usepackage{latexsym}
\usepackage{mathrsfs}
\usepackage{stmaryrd}
\usepackage{xspace}
\usepackage{enumerate, paralist}
\usepackage{graphicx}
\usepackage[all]{xy}
\usepackage{extarrows}
\usepackage{enumitem}

\usepackage[usenames,dvipsnames]{color}

\usepackage{txfonts, pxfonts}

\usepackage{amsthm}
\usepackage{amsmath}

\newtheorem{thm}{Theorem}[section]
 \newtheorem{cor}[thm]{Corollary}
 \newtheorem{lem}[thm]{Lemma}
 \newtheorem{prop}[thm]{Proposition}

\numberwithin{equation}{section}

 \theoremstyle{definition}
  \newtheorem{defn}[thm]{Definition}
 \theoremstyle{remark}
 \newtheorem{rem}[thm]{Remark}
 \newtheorem{ex}[thm]{Example}

\def\R{\mathbb{R}}
\def\N{\mathbb{N}}
\def\B{\mathfrak{B}}
\def\Id{\mathrm{I}}
\def\P{\mathcal{P}}
\def\F{\mathcal{F}}
\def\Nd{\mathcal{N}}

\def\Dom{\{0\} \cup [1,\infty]}
\def\dom{\{0\} \cup [1,\infty)}

\def\lp{\ell^p_E(X)}
\def\lo{\ell^1_E(X)}
\def\lz{\ell^0_E(X)}
\def\linf{\ell^\infty_E(X)}

\def\lpz{\ell^p_{E,0}(X)}

\def\linfz{\ell^\infty_{E,0}(X)}

\def\andx{\quad\mbox{~and~}\quad}
\def\ppg{\mathrm{prop}}
\def\supp{\mathrm{supp}}

\def\Kp{\mathcal{K}_E^p(X)}
\def\Lp{\mathcal{L}_E^p(X)}
\def\Sp{\mathcal{S}_E^p(X)}
\def\Ap{\mathcal{A}_E^p(X)}

\def\Kinf{\mathcal{K}_E^\infty(X)}
\def\Linf{\mathcal{L}_E^\infty(X)}
\def\Sinf{\mathcal{S}_E^\infty(X)}
\def\Ainf{\mathcal{A}_E^\infty(X)}

\def\Kz{\mathcal{K}_E^0(X)}
\def\Lz{\mathcal{L}_E^0(X)}
\def\Sz{\mathcal{S}_E^0(X)}
\def\Az{\mathcal{A}_E^0(X)}

\def\BD{\mathbb{C}[X;E]}

\begin{document}

\title{Extreme cases of limit operator theory on metric spaces}

\author{Jiawen Zhang}
\address{School of Mathematics, University of Southampton, Highfield, SO17 1BJ, United Kingdom.}
\email{jiawen.zhang@soton.ac.uk}

\thanks{Supported by the Sino-British Trust Fellowship by Royal Society.}

\date{}


\baselineskip=16pt

\begin{abstract}
The theory of limit operators was developed by Rabinovich, Roch and Silbermann to study the Fredholmness of band-dominated operators on $\ell^p(\mathbb{Z}^N)$ for $p \in \Dom$, and recently generalised to discrete metric spaces with Property A by \v{S}pakula and Willett for $p \in (1,\infty)$. In this paper, we study the remained extreme cases of $p \in\{0,1,\infty\}$ (in the metric setting) to fill the gaps.
\end{abstract}
\date{\today}
\maketitle

\parskip 4pt


\section{Introduction}

As in linear algebra, a linear operator $A$ on $\ell^p(\mathbb{Z}^N)$ can be regarded as a $\mathbb{Z}^N$-by-$\mathbb{Z}^N$ matrix. We say that $A$ is a \emph{band operator} if all non-zero entries in its matrix sit within a fixed distance from the diagonal, and that $A$ is a \emph{band-dominated operator} if it is a norm-limit of band operators. For each $k \in \mathbb{Z}^N$, the \emph{$k$-shift operator} $V_k: \ell^p(\mathbb{Z}^N) \to \ell^p(\mathbb{Z}^N)$ maps $(x_i)_{i \in \mathbb{Z}^N}$ to $(y_i)_{i \in \mathbb{Z}^N}$ with $y_{i+k}=x_i$. Given a band-dominated operator $A$ and a sequence $(k_m)_{m \in \mathbb N}$ in $\mathbb{Z}^N$ tending to infinity, the sequence $(V_{-k_m}AV_{k_m})_{m \in \mathbb N}$ of shifts of $A$ always contains an entry-wise convergent subsequence by Bolzano-Weierstrass Theorem and a Cantor diagonal argument, and the limit is called a \emph{limit operator} of $A$. The collection of all limit operators of $A$ is called the \emph{operator spectrum} of $A$. See the excellent book \cite{rabinovich2012limit} and the survey paper \cite{chandler2011limit} for relevant references.

Intuitively, the Fredholmness of a band-dominated operator is invariant under compact perturbations such as arbitrary modifications to finitely many entries of the associated matrix, so it should be ``encoded" in the asymptotic behaviours of the matrix. This leads to the central problem in limit operator theory: to study the Fredholmness of a band-dominated operator on the discrete domain $\mathbb{Z}^N$ in terms of its operator spectrum. Major work has been done around it by Lange, Rabinovich, Roch, Roe and Silbermann \cite{lange1985noether, rabinovich2012limit, rabinovich2004fredholm, rabinovich1998fredholm} with several recent contributions by Chandler-Wilde, Lindner, Seidel and others \cite{chandler2011limit, lindner2006infinite, lindner2014affirmative, seidel2014fredholm}.

The following fundamental theorem was proved by Rabinovich, Roch, Silbermann in the case of $p\in (1,\infty)$ and later by Lindner filling the gaps for $p\in \{0,1,\infty\}$.

\begin{thm}[\cite{lindner2006infinite, rabinovich2012limit}]\label{thm1}
Let $A$ be a band-dominated operator on $\ell^p(\mathbb{Z}^N)$, where $p \in \Dom$. Then $A$ is Fredholm if and only if all limit operators of $A$ are invertible and their inverses are uniformly bounded in norm.
\end{thm}

It had been a longstanding question whether the requirement of uniform boundedness can be removed in the above theorem, and was recently answered by Lindner and Seidel affirmatively \cite{lindner2014affirmative}. In some literatures (for example \cite{chandler2011limit, rabinovich2012limit}), Theorem \ref{thm1} is stated in a Banach space valued version, which makes it possible to deal with continuous spaces as well.

Roe realised that it is essentially the coarse geometry of the underlying space that plays a key role in the limit operator theory. In \cite{roe2005band}, he generalised limit operators in the case of $p=2$ to all discrete groups and proved Theorem \ref{thm1} for exact discrete groups.

Following Roe's philosophy, \v{S}pakula and Willett \cite{vspakula2017metric} introduced and studied limit operator theory for any discrete metric space $X$ with Property A. The notion of Property A was introduced by Yu \cite{yu2000coarse}, and it is equivalent to exactness in the case of groups. Having noticed that generally there are no natural shift operators on $X$, they considered the Stone-\v{C}ech boundary of $X$ and constructed a canonical \emph{limit space} associated to each boundary point $\omega$, denoted by $X(\omega)$. Intuitively, $X(\omega)$ captures the geometry of $X$ as one `looks towards infinity in the direction of $\omega$'. Furthermore, for a \emph{rich} band-dominated operator $A$ on $\ell^p(X)$ where $p \in (1,\infty)$, they constructed the limit operator $\Phi_\omega(A)$ as a bounded operator on $\ell^p(X(\omega))$.

Having established the above setup, \v{S}pakula and Willett proved the following generalisation of Theorem \ref{thm1}:

\begin{thm}[\cite{vspakula2017metric}]\label{thm2}
Let $X$ be a strongly discrete metric space with bounded geometry, $p \in (1,\infty)$ and $E$ be a Banach space. Assume that $X$ has Property A. Let $A$ be a rich band-dominated operator on $\ell^p_E(X)$. Then the following are equivalent:
\begin{enumerate}
  \item $A$ is $\P$-Fredholm;
  \item all the limit operators are invertible and their inverses have uniformly bounded norms;
  \item all the limit operators are invertible.
\end{enumerate}
\end{thm}

In this paper, we study the remained cases of $p\in \{0,1,\infty\}$ in the above metric setup, and prove that the above theorem still holds in these cases (Theorem \ref{main thm}). As we will see in the paper, surprisingly some parts of the proof are quite different from the case of $p \in (1,\infty)$, hence new tools and techniques are required. For convenience of the readers, we list the main points as follows.

The first obstacle is to construct limit operators for $p=\infty$. Analogous construction by \v{S}pakula and Willett in the case of $p \in (1,\infty)$ only produces a formal matrix, while we have to verify that it can be realised as the matrix coefficients of a unique bounded operator. This holds trivial when $p<\infty$, while unfortunately, it is not always the case when $p=\infty$ due to the fact that the set of finitely supported vectors are no longer dense in $\ell^\infty(X)$ for infinite $X$.

We work out by proving a density result for $p\in \{0,1,\infty\}$, providing an approach to approximate \emph{rich} band-dominated operators via \emph{rich} band operators. Notice that when $p \in (1,\infty)$, this has already been shown under the additional assumption of Property A on the underlying space \cite[Theorem 6.6]{vspakula2017metric}, which is not necessary for $p \in \{0,1,\infty\}$ from our arguments. We would also like to mention that the approach we develop here also plays a role in a very recent work by \v{S}pakula and the author to study quasi-locality and property A \cite{SZ2018}.

On the other hand, non-density of finitely supported vectors in $\ell^\infty(X)$ also prevents us from using the same tools in \cite{vspakula2017metric} directly to prove some parts of the result. While alternatively, thanks to the imposition of band-domination on operators, we may still consider only finitely supported vectors via a commutant technique.

Another obstacle here is the lack of duality for $p=\infty$. In the case of $p\in (1,\infty)$, as shown in \cite{vspakula2017metric}, properties of operators on $\ell^p(X)$ can be easily transferred to their adjoints since they still act on \emph{$\ell^p$-type} spaces. Unfortunately, this does not apply to the  space $\ell^\infty(X)$. Instead, we borrow dual-space arguments from \cite{chandler2011limit, lindner2003limit}. Roughly speaking, the idea is to consider the double-predual of $\ell^\infty(X)$ (i.e., $c_0(X)$) and properties of operators on $\ell^\infty(X)$ can be characterised nicely via their restrictions on $c_0(X)$, which are much easier to handle.

\textbf{Outline.} For completeness, most of the materials here are written for general $p$ rather than just $p \in \{0,1,\infty\}$. In Section \ref{preliminary}, we recall several classes of operators on $\ell^p$-spaces including band(-dominated) operators; and study when operators can be uniquely determined by their matrix coefficients. In Sections \ref{limit spaces} and \ref{limit operators}, we recall the notion of limit spaces and limit operators in the case of $p\in \dom$. After stating the density result (Theorem \ref{density thm}) whose proof is postponed to Section \ref{density section}, we show in Section \ref{imp for infty} how to implement limit operators when $p=\infty$. Section \ref{construct para section} contains several technical lemmas providing approximations of an operator via its block cutdowns, which is mainly devoted to proving Theorem \ref{density thm} and constructing parametrices later. Finally in Section \ref{main thm section}, we state our main theorem and provide a detailed proof which is divided into several parts.

\emph{Acknowledgments.} First, I would like to thank Kang Li for suggesting this topic and some early discussions. I am also grateful to J\'{a}n \v{S}pakula, Baojie Jiang and Benyin Fu for several illuminating discussions and comments after reading some early drafts of this paper. Finally, I would like to express my sincere gratitude to Graham Niblo and Nick Wright for continuous supports.

\section{Preliminaries}\label{preliminary}
We collect several background notions in operator algebra theory and establish our settings in this section. See \cite{chandler2011limit,rabinovich2012limit} for more information.

For a metric space $(X,d)$, denote by $B(x,R)$ the closed ball in $X$ with radius $R$ and centre $x$. We say that $X$ has \emph{bounded geometry} if for any $R$, the number $\sup_{x\in X} \sharp B(x,R)$ is finite; that $X$ is \emph{strongly discrete} if the set $\{d(x,y):x,y\in X\}$ is a discrete subset of $\R$. We say that `\emph{$X$ is a space}' as shortened for `$X$ is a strongly discrete metric space with bounded geometry' from now on.

\emph{Throughout the paper, always denote $X$ to be a space and $E$ to be a Banach space.}

\subsection{Banach space valued $\ell^p$-space}\label{lp space}
We start with the following notions of Banach space valued $\ell^p$-spaces:
\begin{itemize}
  \item $\lp:=\ell^p(X;E)$ for $p\in [1,\infty)$, which denotes the Banach space of $p$-summable functions from $X$ to $E$ with respect to the counting measure;
  \item $\linf:=\ell^\infty(X;E)$, which denotes the Banach space of bounded functions from $X$ to $E$;
  \item $\lz:=c_0(X;E)$, which denotes the Banach space of continuous functions from $X$ to $E$ vanishing at infinity.
\end{itemize}

For a Banach space $Y$, denote by $\B(Y)$ the Banach algebra of all bounded linear operators on $Y$. For $p \in \Dom$, $\rho: C_b(X) \to \B(\lp)$ is a representation defined by point-wise multiplication. To simplify notations, we write $f \xi$ instead of $\rho(f)(\xi)$ for $f \in C_b(X)$ and $\xi \in \lp$.

Denote by $\F$ the set of all finite subsets in $X$, equipped with the order by inclusion. For any $F \in \F$, define an operator $P_F:=\rho(\chi_F)$ on $\lp$ where $\chi_F$ denotes the characteristic function of $F$, and set $Q_F:=\Id-P_F$. Clearly, the net $\P:=\{P_F\}_{F \in \F}$ satisfies the following conditions for $p \in \Dom$:
\begin{enumerate}
  \item[i)] $\sup_F \|P_F\xi\| = \|\xi\|$ for any $\xi \in \lp$;
  \item[ii)] $P_FP_{F'}=P_F=P_{F'}P_F$ for any $F \subseteq F'$.
\end{enumerate}
Furthermore, denote by $\lpz$ the closure of $\bigcup_{F \in \F}P_F(\lp)$ in $\lp$. We point out the following elementary observation:

\begin{lem}\label{Y0}
$\lpz=\lp$ for $p\in \dom$, while $\linfz=\lz$.
\end{lem}

\subsection{Classes of Operators}\label{class of op}
Here we recall several classes of operators on $\lp$. Let $p \in \{0\} \cup [1,\infty]$ be fixed.

\begin{defn}
We define subsets $\Kp$, $\Lp$ and $\Sp$ of $\B(\lp)$ as follows:
\begin{enumerate}
  \item $K \in \Kp$ \emph{if and only if}
        $$\lim_{F\in \F}\|KQ_F\|=0 \andx \lim_{F\in \F}\|Q_FK\|=0.$$
        Elements in $\Kp$ are called \emph{$\P$-compact}.
  \item $T \in \Lp$ \emph{if and only if} for any $F'\in \F$, we have
        $$\lim_{F\in \F}\|P_{F'}TQ_F\|=0 \andx \lim_{F\in \F}\|Q_FTP_{F'}\|=0.$$
  \item $T \in \Sp$ \emph{if and only if} for any $F'\in \F$, we have
        $$\lim_{F\in \F}\|P_{F'}TQ_F\|=0.$$
\end{enumerate}
\end{defn}

Notice that $\Kp, \Lp$ and $\Sp$ are also denoted by $\mathcal{K}(Y,\P)$, $\mathcal{L}(Y,\P)$ and $S(Y)$ for $Y=\lp$ respectively in some literatures, for example \cite{chandler2011limit, lindner2006infinite, rabinovich2012limit}. Clearly $\P \subseteq \Kp \subseteq \Lp \subseteq \Sp$, and $\Kp, \Lp, \Sp$ are linear subspaces of $\B(\lp)$. Furthermore, we have:

\begin{lem}[\cite{chandler2011limit}]
$\Kp, \Lp$ and $\Sp$ are Banach subalgebras of $\B(\lp)$.
\end{lem}

\begin{proof}
It suffices to show the closeness under multiplication, which is obvious for $\Kp$. We only prove for $\Sp$, and the argument is similar for $\Lp$. Given $S,T \in \Sp$, for any $F' \in \F$ and $\epsilon>0$, there exists $G\in \F$ such that $\|P_{F'}SQ_G\|<\frac{\epsilon}{2\|T\|}$. For such $G$, there exists some $F \in \F$ such that $\|P_GTQ_F\|<\frac{\epsilon}{2\|S\|}$. Hence we have:
$$\|P_{F'}STQ_F\| \leq \|P_{F'}SQ_GTQ_F\| + \|P_{F'}SP_GTQ_F\| < \frac{\epsilon}{2\|T\|}\cdot \|T\| + \|S\| \cdot \frac{\epsilon}{2\|S\|}=\epsilon,$$
and we finish the proof.
\end{proof}

\begin{lem}[\cite{chandler2011limit}]\label{ideal}
$\Kp$ is a two-sided ideal in $\Lp$, and $\Lp$ is the largest subalgebra of $\B(\lp)$ with this property.
\end{lem}

\begin{proof}
For any $S\in \Kp$ and $T \in \Lp$, we know $\|Q_FST\| \to 0$ since $\|Q_FS\| \to 0$. On the other hand, for any $\epsilon>0$, there exists $G \in \F$ such that $\|SQ_G\|<\frac{\epsilon}{2\|T\|}$. For such $G$, there exists some $F \in \F$ such that $\|P_GTQ_F\|<\frac{\epsilon}{2\|S\|}$. Hence,
$$\|STQ_F\|\leq \|SQ_GTQ_F\|+\|SP_GTQ_F\| < \frac{\epsilon}{2\|T\|}\cdot \|T\|+\|S\|\cdot \frac{\epsilon}{2\|S\|}=\epsilon.$$
So $ST\in \Kp$, and similarly $TS \in \Kp$. For the second statement, given $T\in \B(\lp)$ satisfying that $ST$ and $TS$ sit in $\Kp$ for any $S \in \Kp$, then $T \in \Lp$ since $\P \subseteq \Kp$.
\end{proof}

\begin{rem}
Recall that by \cite[Proposition 1.1.9]{rabinovich2012limit}, $\Lp$ is closed under taking inverses. Also notice that by \cite[Chapter 3]{chandler2011limit}, $\Sp$ is exactly the set of all sequentially continuous operators on $\lp$, which explains the terminology.
\end{rem}

\begin{lem}\label{inv}
For any $T \in \Lp$, the subspace $\lpz$ is invariant under $T$.
\end{lem}

\begin{proof}
Clearly, $\lpz=\{v\in \lp: \lim_{F \in \F}\|Q_Fv\| = 0\}$. Hence for any $v \in \lpz$ and $\epsilon>0$, there exists $F \in \F$ such that $\|Q_Fv\|<\epsilon$. Since $T\in \Lp$, there exists $G \in \F$ such that $\|Q_GTP_F\| < \epsilon$. Combining them together, we have
$$\|Q_GTv\| \leq \|Q_GTP_Fv\| + \|Q_GTQ_Fv\| < \epsilon \|v\| + \epsilon\|T\|.$$
Taking $\epsilon \to 0$, we have $\|Q_GTv\| \to 0$, which implies that $Tv \in \lpz$.
\end{proof}

The notions of $\P$-Fredholmness and invertibility at infinity are fundamental in limit operator theory, whose definitions are recalled as follows:
\begin{defn}[\cite{rabinovich2012limit, vspakula2017metric}]
An operator $A\in \Lp$ is \emph{invertible at infinity} if there exists $B \in \B(\lp)$ such that $AB-\Id$ and $BA-\Id$ belong to $\Kp$. $A \in \Lp$ is \emph{$\P$-Fredholm} if the above $B$ can be chosen in $\Lp$.
\end{defn}

\subsection{Matrix Coefficients}\label{matrix coeff}
For $p \in \Dom$ and $x\in X$, let $E_x:E \to \lp$ and $R_x: \lp \to E$ be the extension and restriction operators, defined by:
\begin{equation*}
E_x(e)(y)=
\begin{cases}
  ~e, & y=x, \\
  ~0, & \mbox{otherwise};
\end{cases}
\end{equation*}
and $R_x(\xi)=\xi(x)$ where $\xi \in \lp$. For $T \in \B(\lp)$ and $x,y\in X$, define its \emph{$xy$-matrix coefficient} to be $T_{xy}=R_xTE_y \in \B(E)$.

A natural question is that to what extent do matrix coefficients determine the operator itself? It is obvious that for $\xi \in \bigcup_{F \in \F}P_F(\lp)$, $T\xi$ can be reformulated as
$$(T\xi)(x)=R_xT\big(\sum_{y \in X}E_y\xi(y)\big)=\sum_{y \in X}T_{xy}\xi(y)$$
for each $x\in X$. Passing to the closure and by Lemma \ref{Y0}, we obtain that for $p\in \dom$, an operator on $\lp$ can be determined by its matrix coefficients. Furthermore, in this case we have that for any $T\in \B(\lp)$,
\begin{equation}\label{norm p fin}
  \|T\|=\sup_{F\in \F}\|P_FTP_F\|.
\end{equation}
Conversely, we have the following result, whose proof is straightforward.

\begin{lem}\label{matrix to op for finite p}
Given $p \in \dom$ and an $X$-by-$X$ matrix $A=[a_{xy}]_{x,y \in X}$ with entries in $\B(E)$. Suppose there exists $M>0$ such that: For any finite $F \subseteq X$, the operator $A_F$ on $\ell^p(F)$ defined by multiplication by the finite submatrix $[a_{xy}]_{x,y \in F}$ has norm at most $M$. Then the multiplication by $A$ provides a well-defined linear operator on $\lp$ with norm $\sup_{F \in \F} \|A_F\|$.
\end{lem}

Unfortunately, since finitely supported vectors are no longer dense in $\linf$ generally, things become complicated when $p=\infty$. However, we still have:

\begin{lem}\label{op to matrix coeff}
For $T \in \Sinf$, $\xi \in \linf$ and $x\in X$, the series $\sum_{y \in X}T_{xy}\xi(y)$ converges to $(T\xi)(x)$, where $T_{xy}$ is the $xy$-matrix coefficient of $T$. Furthermore, $\|T\|=\sup_{F\in \F}\|P_FTP_F\|$.
\end{lem}

The proof follows directly from the definition, hence omitted.

\subsection{Band (Dominated) Operators}\label{BDO}
Now we introduce our main objects.
\begin{defn}[\cite{chandler2011limit}, \cite{vspakula2017metric}]
Let $A=[A_{xy}]_{x,y \in X}$ be an $X$-by-$X$ matrix with entries in $\B(E)$. $A$ is a \emph{band operator (on $X$)} if
\begin{enumerate}
  \item the norms $\|A_{xy}\|$ are uniformly bounded for $x,y\in X$;
  \item the \emph{propagation} of $A$, defined by $\ppg(A):=\sup\{d(x,y): A_{xy} \neq 0\}$,
  is finite.
\end{enumerate}
Let $\BD$ denote the collection of all band operators on $X$.
\end{defn}

\begin{ex}[\cite{vspakula2017metric}]\label{mul and pt}
\begin{enumerate}
  \item Let $f: X \to \B(E)$ be a bounded function. Then the diagonal matrix defined by
        \begin{equation*}
            A_{xy}=
            \begin{cases}
              ~f(x), & y=x, \\
              ~0, & y\neq x
            \end{cases}
        \end{equation*}
      is a band operator of propagation 0, called a \emph{multiplication operator}.
  \item Let $D,R$ be subsets of $X$, and $t:D \to R$ be a bijection such that $\sup_{x\in D}d(x,t(x))$ is finite. Define a matrix $V$ by
      \begin{equation*}
            V_{yx}=
            \begin{cases}
              ~\Id_E, & x\in D \mbox{~and~} y=t(x), \\
              ~0, & \mbox{otherwise}.
            \end{cases}
        \end{equation*}
      Then $V$ is a band operator, called a \emph{partial translation operator}.
\end{enumerate}
\end{ex}

The above two classes of operators indeed generate $\BD$ as an algebra. The following lemma is very well-known:
\begin{lem}[\cite{vspakula2017metric}]\label{dec of BO}
Let $A\in\BD$ have propagation at most $r$ and $N=\sup_{x\in X} \sharp B(x,r)$. Then there exist multiplication operators $f_1,\ldots,f_N$ with $\|f_k\| \leq \sup_{x,y}\|A_{xy}\|$, and partial translation operators $V_1,\ldots,V_N$ of propagation at most $r$ such that:
$$A=\sum_{k=1}^N f_kV_k.$$
\end{lem}

Consequently, a band operator can always be realised as a bounded linear operator on any $\ell^p$-space:
\begin{cor}
Let $A \in \BD$ have propagation at most $r$, and $p \in \Dom$. Then the operator on $\lp$ defined by matrix multiplication by $A$ is bounded and belongs to $\Lp$, with norm at most $\sup_{x,y}\|A_{xy}\| \cdot \sup_{x\in X} \sharp B(x,r)$.
\end{cor}

From now on, $\BD$ is regarded as a subalgebra of $\Lp$ for $p \in \Dom$. Conversely, an operator $A \in \Lp$ is a band operator if there exists some $r>0$ such that $A_{xy}=0$ for $d(x,y)>r$. Taking norm closures, we have:

\begin{defn}
For $p \in \Dom$, the norm-closure of $\BD$ in $\B(\lp)$ is denoted by $\Ap$. Elements in $\Ap$ are called \emph{band-dominated operators}.
\end{defn}

Combined with Lemma \ref{ideal}, we end up with the following inclusions:
\begin{lem}[\cite{chandler2011limit}]
For the above subalgebras, we have $\Kp \subseteq \Ap \subseteq \Lp \subseteq \Sp$. Furthermore, $\Kp$ is a two-sided ideal in $\Lp$, hence a two-sided ideal in $\Ap$ as well.
\end{lem}

\section{Limit spaces and limit operators}

Throughout this section, $X$ is a space, $\beta X$ is the associated Stone-\v{C}ech compactification and $\partial X=X \setminus \beta X$ is the boundary. We use the classical identification for $\beta X$ in terms of ultrafilters on $X$. More details can be found in literatures, for example \cite[Appendix A]{brown2008cstar} or \cite[Appendix A]{vspakula2017metric}.

\subsection{Limit spaces}\label{limit spaces}
The notion of limit spaces for general metric spaces were introduced by \v{S}pakula and Willett. In the case of groups, all limit spaces are isometric to the group itself due to homogeneity. Here we only collect the notions and results required later, while more details and illuminating examples can be found in \cite[Section 3]{vspakula2017metric}.

Recall from Example \ref{mul and pt} that a function $t: D \to R$ with $D,R \subseteq X$ is called a \emph{partial translation} if $t$ is a bijection from $D$ to $R$, and $\sup_{x\in X}d(x,t(x))$ is finite.

\begin{defn}[\cite{vspakula2017metric}]
Fix an ultrafilter $\omega \in \beta X$. A partial translation $t:D \to R$ on $X$ is \emph{compatible with $\omega$} if $\omega(D)=1$. In this case, regarding $t$ as a function from $D$ to $\beta X$, define
$$t(\omega):=\lim_\omega t \in \beta X.$$
An ultrafilter $\alpha \in \beta X$ is \emph{compatible with $\omega$} if there exists a partial translation $t$ which is compatible with $\omega$ and $t(\omega)=\alpha$.
\end{defn}

\begin{rem}\label{limit comp}
As pointed out in \cite{vspakula2017metric}, an ultrafilter $\alpha \in \beta X$ is compatible with $\omega \in \beta X$ \emph{if and only if} there exists a partial translation $t:D \to R$ such that $\omega(D)=1$ and $\alpha(S)=1$ iff $\omega(t^{-1}(S \cap R))=1$ for any $S \subseteq X$. Therefore compatibility is an equivalence relation.
\end{rem}

Let us recall the following uniqueness statement:
\begin{lem}[\cite{vspakula2017metric}]\label{unique}
Let $\omega$ be an ultrafilter on $X$, and $t: D_t \to R_t$ and $s: D_s \to R_s$ be two partial translations compatible with $\omega$ such that $s(\omega)=t(\omega)$. Then if
$$D:=\{x \in D_t \cap D_s: t(x)=s(x)\},$$
we have that $\omega(D)=1$.
\end{lem}

\begin{defn}[\cite{vspakula2017metric}]
Fix an ultrafilter $\omega$ on $X$. Write $X(\omega)$ for the collection of all ultrafilters on $X$ compatible with $\omega$. A \emph{compatible family for $\omega$} is a collection of partial translations $\{t_\alpha\}_{\alpha \in X(\omega)}$ such that each $t_\alpha$ is compatible with $\omega$ and $t_\alpha(\omega)=\alpha$.
\end{defn}

A metric can be imposed on $X(\omega)$ as follows:
\begin{prop}[\cite{vspakula2017metric}]
Fix an ultrafilter $\omega$ on $X$, and a compatible family $\{t_\alpha\}_{\alpha \in X(\omega)}$. Define a function $d_{\omega}: X(\omega) \times X(\omega) \to [0,\infty)$ by
$$d_{\omega}(\alpha, \beta):=\lim_{x\to \omega}d(t_\alpha(x),t_\beta(x)).$$
Then $d_\omega$ is a metric on $X(\omega)$ that does not depend on the choice of $\{t_{\alpha}\}$. Moreover,
$$\{d_\omega(\alpha, \beta): \alpha, \beta \in X(\omega)\} \subseteq \{d(x,y): x,y \in X\}$$
and
$$\max_{\alpha \in X(\omega)} \sharp B_{X(\omega)}(\alpha,r) \leq \max_{x\in X} \sharp B_X(x,r).$$
\end{prop}

\begin{defn}[\cite{vspakula2017metric}]
For each non-principal ultrafilter $\omega$ on $X$, the metric space $(X(\omega),d_{\omega})$ is called the \emph{limit space} of $X$ at $\omega$.
\end{defn}

It is shown in \cite[Proposition 3.9]{vspakula2017metric} that $X(\omega)$ does not depend on the choice of $\omega$ in the sense that for any $\alpha \in X(\omega)$, $X(\alpha)=X(\omega)$. Now we recall the following result stating that the local geometry of $X$ can be captured by that of its limit space.

\begin{prop}[\cite{vspakula2017metric}]\label{local isom}
Let $\omega$ be a non-principal ultrafilter on $X$, and $\{t_\alpha:D_\alpha \to R_\alpha\}$ a compatible family for $\omega$. Then for each finite $F \subseteq X(\omega)$, there exists a subset $Y \subseteq X$ with $\omega(Y)=1$, such that for each $y\in Y$, there is a finite subset $G(y) \subseteq X$ such that the map
$$f_y: F \to G(y), \alpha \mapsto t_\alpha(y)$$
is a surjective isometry. Such a collection $\{f_y\}_{y\in Y}$ is called a \emph{local coordinate systerm } for $F$, and the maps $f_y$ are called \emph{local coordinates}.

Furthermore, if $F$ is a metric ball $B(\omega,r)$, then there exist $Y \subseteq X$ with $\omega(Y)=1$ and a local coordinate system $\{f_y: F \to G(y)\}_{y\in Y}$ such that each $G(y)$ is the ball $B(y,r)$.
\end{prop}

Interesting examples of limits spaces can be found in \cite[Example 3.14]{vspakula2017metric}.

\subsection{Limit operators}\label{limit operators}
Now we introduce the notion of limit operators on metric spaces. The case of $p\in (1,\infty)$ was studied by \v{S}pakula and Willett \cite{vspakula2017metric}, while for the convenience to the readers, they are stated here as well. We start with the following condition of richness for an operator, which is designed to ensure that limit operators always exist.

\begin{defn}
Let $\omega$ be a non-principal ultrafilter on $X$. For $p \in \Dom$, a band-dominated operator $A$ in $\Ap$ is said to be \emph{rich at $\omega$}, if for any pair of partial translations $t,s$ compatible with $\omega$, the limit
$$\lim_{x \to \omega}A_{t(x)s(x)}$$
exists under the norm topology on $\B(E)$. Denote by $\Ap^{\$,\omega}$ the collection of all band-dominated operators rich at $\omega$.

If $A$ is rich at $\omega$ for all $\omega \in \partial X$, it is said to be \emph{rich}. Denote by $\Ap^{\$}$ the collection of all rich band-dominated operators.
\end{defn}

\begin{defn}
Let $\omega$ be a non-principal ultrafilter on $X$, $p \in \Dom$ and $A$ be a band-dominated operator on $\lp$ rich at $\omega$. Fix a compatible family $\{t_\alpha\}_{\alpha \in X(\omega)}$ for $\omega$. The \emph{limit operator of $A$ at $\omega$}, denoted by $\Phi_\omega(A)$, is an $X(\omega)$-by-$X(\omega)$ indexed matrix with entries in $\B(E)$ defined by
$$\Phi_\omega(A)_{\alpha\beta}:=\lim_{x \to \omega}A_{t_{\alpha}(x)t_{\beta}(x)}.$$
\end{defn}

From Lemma \ref{unique}, we know the above definition is proper:
\begin{lem}
Let $\omega$ be a non-principal ultrafilter on $X$, $p \in \Dom$ and $A$ a band-dominated operator on $\lp$ rich at $\omega$. Then the limit operator $\Phi_\omega(A)$ does not depend on the choice of compatible family of $\omega$.
\end{lem}

Up till now, $\Phi_\omega(A)$ is only an abstractly defined infinite matrix, rather than an operator on any space. We will follow the way in \cite{vspakula2017metric} to make it concrete in the case of $p\in \dom$ as follows, while leave the case of $p=\infty$ to the next subsection where more technical tools are developed.

\begin{prop}\label{local isom op}
Let $\omega$ be a non-principal ultrafilter on $X$, $p \in \Dom$ and $A$ a band-dominated operator on $\lp$ rich at $\omega$. Let $F$ be a finite subset of $X(\omega)$ and $\epsilon>0$. Let $\{t_\alpha\}$ be a compatible family of partial translations for $\omega$. Then there exists a local coordinate system $\{f_y:F \to G(y)\}_{y \in Y}$ such that for each $y \in Y$, if
$$U_y: \ell^p(F) \to \ell^p(G(y)), ~~(U_y\xi)(x):=\xi(f_y^{-1}(x))$$
is the linear isometry induced by $f_y$, then we have:
$$\|U_y^{-1}P_{G(y)}AP_{G(y)}U_y-P_F\Phi_\omega(A)P_F\| < \epsilon,$$
where $P_F\Phi_\omega(A)P_F$ is regarded as a finite $F$-by-$F$ matrix with entries in $\B(E)$, acting on $\ell^p_E(F)$ by matrix multiplication.
\end{prop}

The proof follows directly from Proposition \ref{local isom} and is the same as that of \cite[Proposition 4.6]{vspakula2017metric}, hence omitted. Consequently, we have the following concrete implementation for limit operators in the case of $p \in \dom$. The proof is just a combination of Lemma \ref{matrix to op for finite p} and the above proposition, hence omitted.
\begin{cor}\label{p finite}
Let $\omega$ be a non-principal ultrafilter on $X$, $p \in \dom$ and $A$ a band-dominated operator on $\lp$ rich at $\omega$. Then the matrix multiplication by $\Phi_\omega(A)$ defines a bounded operator on $\ell^p(X(\omega))$ with norm at most $\|A\|$.
\end{cor}

Therefore in the case of $p \in \dom$, we may regard limit operators as concrete bounded linear operators on $\ell^p$-type spaces. See \cite[Section 4]{vspakula2017metric} for illuminating examples, and \cite[Appendix B]{vspakula2017metric} for the comparisons with the classical limit operators in the case of groups.

\subsection{Implementing limit operators for $p=\infty$}\label{imp for infty}

As we point out in Section \ref{matrix coeff}, operators on $\linf$ may not be uniquely determined by their matrix coefficients in general. Hence, to realise limit operators via matrices in this case, we need an extra density result as follows. It is stated for $p \in \{0,1,\infty\}$ due to further uses.

\begin{thm}\label{density thm}
Let $X$ be a space, $\omega$ be a non-principal ultrafilter on $X$ and $p \in \{0,1,\infty\}$. The set of band operators on $\lp$ rich at $\omega$ are dense in $\Ap^{\$,\omega}$; and the set of rich band operators on $\lp$ are dense in $\Ap^{\$}$.
\end{thm}

The proof is postponed to Section \ref{density section} after we develop the tool to construct operators via blocks. We point out that the density result also holds for $p \in (1,\infty)$ under the extra assumption of Property A on $X$ \cite[Theorem 6.6]{vspakula2017metric}, while surprisingly as we will see, it holds generally for $p\in\{0,1,\infty\}$.

Now we use Theorem \ref{density thm} to show that $\Phi_\omega(A)$ can be implemented as a bounded operator when $p=\infty$. We start with the following lemma.

\begin{lem}\label{limit on BD}
Let $\omega$ be a non-principal ultrafilter on $X$, and $A \in \BD$ be a band operator on $\linf$ rich at $\omega$. Then $\Phi_\omega(A)$ is a band operator whose propagation does not exceed that of $A$, and $\|\Phi_\omega(A)\| \leq \|A\|$.
\end{lem}

\begin{proof}
The claim on propagation follows directly from that of \cite[Theorem 4.10(2)]{vspakula2017metric}, hence omitted. Concerning norms, note that $\Phi_\omega(A) \in \mathbb{C}[X(\omega);E] \subseteq \mathcal{S}^\infty_E(X(\omega))$. Hence by Lemma \ref{op to matrix coeff}, $\|\Phi_\omega(A)\|=\sup_{F \in \F} \|P_F\Phi_\omega(A)P_F\|$. On the other hand, by lemma \ref{local isom op}, $\|P_F\Phi_\omega(A)P_F\| \leq \|A\|$ for any finite $F\subseteq X$. So the lemma holds.
\end{proof}

\begin{prop}\label{impl of limit}
Let $X$ be a space and $\omega$ a non-principal ultrafilter on $X$. Suppose $A \in \Ainf$ is rich at $\omega$, and $\Phi_\omega(A)$ is the associated $X(\omega)$-by-$X(\omega)$ matrix. Then there exists a unique operator in $\mathcal{S}^\infty_E(X(\omega))$ with matrix coefficients $\{\Phi_\omega(A)_{\alpha\beta}\}_{\alpha,\beta \in X(\omega)}$ and norm at most $\|A\|$. Furthermore, this operator belongs to $\mathcal{A}^\infty_E(X(\omega))$.
\end{prop}

\begin{proof}
By Theorem \ref{density thm}, there exist band operators $\{A_n\}_{n \in \N}$ on $\linf$ rich at $\omega$ and converging to $A$. By Lemma \ref{limit on BD}, each $\Phi_\omega(A_n)$ is a band operator and $\{\Phi_\omega(A_n)\}_{n \in \N}$ is Cauchy, hence it converges to some $T$ in $\mathcal{A}^\infty_E(X(\omega))$. For any $\alpha,\beta \in X(\omega)$ and $\epsilon>0$, there exists some $N$ such that $\|\Phi_\omega(A_N)_{\alpha \beta}-T_{\alpha\beta}\| < \epsilon/3$ and $\|A_N-A\|<\epsilon/3$. By definition, if we set
$$Y=\{x\in X: \|(A_N)_{t_\alpha(x)t_\beta(x)}-\Phi_\omega(A_N)_{\alpha \beta}\|< \epsilon/3\},$$
then $\omega(Y)=1$. Hence for $x\in Y$, we have
\begin{eqnarray*}
\|A_{t_\alpha(x)t_\beta(x)}-T_{\alpha\beta}\| &\leq& \|A_{t_\alpha(x)t_\beta(x)}-(A_N)_{t_\alpha(x)t_\beta(x)}\| + \|(A_N)_{t_\alpha(x)t_\beta(x)}-\Phi_\omega(A_N)_{\alpha \beta}\| \\
&&~+~ \|\Phi_\omega(A_N)_{\alpha \beta}-T_{\alpha\beta}\| \\
&<& \epsilon/3 + \epsilon/3 + \epsilon/3 =\epsilon.
\end{eqnarray*}
Therefore,
$$\omega\big(\{x\in X: \|A_{t_\alpha(x)t_\beta(x)}-T_{\alpha\beta}\|< \epsilon\}\big)=1$$
for any $\epsilon$, which implies that
$$\Phi_\omega(A)_{\alpha\beta}=\lim_{x\to \omega}A_{t_\alpha(x)t_\beta(x)}=T_{\alpha\beta}.$$
So $T$ is the required operator and it is unique by Lemma \ref{op to matrix coeff}.
\end{proof}

\begin{cor}\label{bd for 01infty}
Let $X$ be a space, $p\in \{0,1,\infty\}$ and $A \in \Ap$. Then for any $\omega \in \partial X$, the limit operator $\Phi_\omega(A)$ belongs to $\mathcal{A}^p_E(X(\omega))$.
\end{cor}

Now we collect several important properties of limit operators already obtained as follows. The proof is very similar to \cite[Theorem 4.10]{vspakula2017metric}, combined with Proposition \ref{impl of limit} and Corollary \ref{bd for 01infty}, hence omitted.

\begin{prop}\label{limit op homo}
Let $\omega$ be a non-principal ultrafilter on $X$ and $p\in\Dom$. Then the map
$$\Phi_\omega: \Ap^{\$,\omega} \to \B(\ell^p_E(X(\omega)))$$
has the following properties:
\begin{enumerate}
  \item $\Phi_\omega$ is contractive.
  \item $\Phi_\omega$ takes band operators to band operators and does not increase propagations.
  \item $\Phi_\omega$ is a homomorphism.
  \item $\Phi_\omega$ has image in $\mathcal{A}^p_E(X(\omega))$ for $p\in \{0,1,\infty\}$.
\end{enumerate}
\end{prop}

\begin{defn}
For $p\in \Dom$, let $A \in \Ap^{\$}$ be a rich band-dominated operator. The collection
$$\sigma_{op}(A):=\{\Phi_\omega(A) \in \B(\ell^p_E(X(\omega))): \omega \in \partial X\}$$
is called the \emph{operator spectrum} of $A$.
\end{defn}

Later in Theorem \ref{main thm}, we will detect the properties of $\P$-Fredholmness and invertibility at infinity for a rich band-dominated operator in terms of its operator spectrum for $p \in \Dom$.

\section{Partition of unity and the density theorem}\label{partition of unity}

In this section, as promised, we develop the technique to approximate operators via block cutdowns, especially in the case of $p \in \{0,1,\infty\}$. They are not only crucial to prove the density result (Theorem \ref{density thm}) above, but also useful to construct parametrices in later proof of our main theorem.

Throughout this section, $X$ is a space and $E$ is a Banach space. For $p \in [1,\infty)$, set $q \in (1,\infty]$ to be the conjugate exponent of $p$, i.e., $1/p+1/q=1$.

\subsection{Property A, partition of unity, and constructing operators}\label{construct para section}

Property A was first introduced by Yu \cite{yu2000coarse}, and since then it has been shown to be equivalent to a lot of other properties. Here we shall use the formulation in terms of the existence of partitions of unity with small variation \cite[Theorem 1.2.4]{willett2006some}, which helps us to cut an operator into blocks.

\begin{defn}[\cite{vspakula2017metric}]
For $p \in [1,\infty)$, a \emph{metric $p$-partition of unity} on $X$ is a collection $\{\phi_i: X \to [0,1]\}_{i\in I}$ of functions on $X$ such that:
\begin{enumerate}
  \item There exists $N \in \mathbb N$ such that for each $x \in X$, at most $N$ of $\phi_i(x)$ are non-zero.
  \item $\{\phi_i\}_{i\in I}$ have uniformly bounded supports, i.e.,
  $$\sup_{i\in I} (\mathrm{diam} (\{x \in X:\phi_i(x) \neq 0\})) < \infty.$$
  \item For each $x \in X$, $\sum_{i \in I} \phi_i(x)^p=1$.
\end{enumerate}
Let $r,\epsilon>0$. A metric $p$-partition of unity $\{\phi_i\}_{i\in I}$ is said to have $(r,\epsilon)$-\emph{variation} if for any $x,y \in X$ with $d(x,y) \leq r$, we have
$$\sum_{i \in I} |\phi_i(x)-\phi_i(y)|^p < \epsilon^p.$$
A space $X$ is said to have \emph{Property A} if for any $r,\epsilon>0$, there exists a metric $p$-partition of unity with $(r,\epsilon)$-variation.
\end{defn}

A lot of interesting spaces and groups are known to have Property A. For example, amenable groups \cite{yu2000coarse}, hyperbolic spaces \cite{roe2005hyperbolic}, CAT(0) cube complexes with finite dimension \cite{brodzki2009property} and all linear groups over any field \cite{guentner2005the}.

To deal with the extreme cases that $p\in \{0,1,\infty\}$, we introduce the following notion of dual family.

\begin{defn}
Let $\{\phi_i\}_{i \in I}$ be a metric $1$-partition of unity on $X$. A \emph{dual family} of $\{\phi_i\}_{i \in I}$ is defined to be a collection $\{\psi_i:X \to [0,1]\}_{i \in I}$ satisfying $\psi_i|_{\supp(\phi_i)}\equiv 1$ and there exists some $R>0$ such that $\supp(\psi_i) \subseteq \Nd_R(\supp(\phi_i))$ for any $i \in I$. A dual family $\{\psi_i\}_{i \in I}$ is \emph{$L$-Lipschitz}, if each $\psi_i$ is $L$-Lipschitz.
\end{defn}

Clearly, for any metric $1$-partition of unity on $X$ and any $L>0$, $L$-Lipschitz dual family always exists. And there exists $N' \in \mathbb N$ such that for each $x \in X$, at most $N'$ of $\psi_i(x)$ are non-zero since $X$ has bounded geometry.

Furthermore, to deal with the case of $p=\infty$, we also need the following topology on $\B(\linf)$.
\begin{defn}
Let $\{B_i\}_{i \in I}$ be a collection of bounded linear operators on $\linf$. If for any $v \in \linf$, the series $\sum_{i\in I}B_iv$ converges point-wise to a vector in $\linf$, and the map
$$\linf \to \linf, \quad v \mapsto \sum_{i \in I} B_i v $$
is a bounded linear operator, then we say $\sum_{i \in I} B_i$ \emph{converges point-wise strongly}.
\end{defn}

Clearly by definition, for a metric $1$-partition of unity $\{\phi_i\}_{i \in I}$, $\sum_{i \in I}\phi_i$ converges point-wise strongly to the identity on $\linf$. For compositions, we have the following elementary observation:
\begin{lem}\label{pwc lem1}
Let $\{B_i\}_{i \in I}$ be a collection of linear bounded operators on $\linf$ such that $\sum_{i \in I}B_i$ converges point-wise strongly, and $A \in \B(\linf)$. Then:
\begin{enumerate}
  \item $\sum_{i \in I}B_iA$ converges point-wise strongly, and we have $\big( \sum_{i \in I} B_i\big)A=\sum_{i \in I}B_iA$.
  \item If $A$ is a band operator, then $\sum_{i \in I}AB_i$ converges point-wise strongly, and we have $A\big( \sum_{i \in I} B_i\big)=\sum_{i \in I}AB_i$.
\end{enumerate}
\end{lem}

(1) holds directly by definition and for (2), we only need to verify in the case that $A$ is a multiplication operator or a partial translation by Lemma \ref{dec of BO}. Both of them are straightforward, hence we omit the proof.

Now we are in the position to construct operators via blocks. To unify the statements, we take the liberty of calling a \emph{metric $p$-partition of unity} for $p\in \{0,\infty\}$ instead of a metric $1$-partition of unity.

\begin{lem}\label{cstrct op}
For $p \in \Dom$, let $\{\phi_i\}_{i\in I}$ be a metric $p$-partition of unity on $X$, and $\{\psi_i\}_{i \in I}$ be a dual family of $\{\phi_i\}_{i \in I}$ when $p\in \{0,1,\infty\}$. Let $J \subseteq I$ and given a collection of bounded linear operators $\{B_i\}_{i \in J}$ on $\lp$ such that $M:=\sup_{i}\|B_i\|$ is finite.
\begin{enumerate}
  \item When $p< \infty$, consider the following:
    \begin{enumerate}
       \item $\sum_{i \in J} \phi_i^{p/q}B_i\phi_i$ if $p \in (1,\infty)$;
       \item $\sum_{i \in J} \phi_i B_i \psi_i$ if $p=0$;
       \item $\sum_{i \in J} \psi_i B_i \phi_i$ if $p=1$.
    \end{enumerate}
    Each of them converges strongly to a band operator of norm at most $M$ on $\lp$.
  \item When $p=\infty$, consider $\sum_{i \in J} \phi_i B_i \psi_i$. It converges point-wise strongly to a band operator of norm at most $M$ on $\linf$.
\end{enumerate}
\end{lem}

\begin{proof}
Case (1)(a) is from \cite[Lemma 6.3]{vspakula2017metric}, hence omitted. We start with Case (1)(b). For any $v\in \lz$ with finite support, only finitely many of the terms in $\sum_{i \in J} \phi_i B_i \psi_iv$ is non-zero, hence produces a well-defined vector in $\lz$. For any $x\in X$, we have
\begin{eqnarray}\label{EQ1}
  \big\| \big(\sum_{i \in J} \phi_i B_i \psi_iv\big)(x) \big\|_E &\leq & \sum_{i \in J}\phi_i(x)\cdot\big\|\big(B_i\psi_i v\big)(x)\big\|_E
   \leq  \sum_{i \in J}\phi_i(x)\cdot\|B_i\psi_i v\|_\infty \nonumber\\
   &\leq & \sum_{i \in J}\phi_i(x)\cdot M \cdot \|v\|_\infty
   \leq  M \cdot \|v\|_\infty.
\end{eqnarray}
Due to density of finitely supported vectors in $\lz$, \eqref{EQ1} holds for all $v\in \lz$.

We move on to Case (1)(c): Again, for any $v \in \lo$ with finite support, $\sum_{i \in J} \psi_i B_i \phi_iv$ is a finite sum. For any unit vector $w$ in $\ell^\infty_{E^*}(X) \cong \lo^*$, we have
\begin{eqnarray*}
  \big| \big\langle \sum_{i \in J} \psi_i B_i \phi_iv, w \big\rangle \big| &\leq & \sum_{i\in J} |\langle B_i\phi_iv, \psi_i w\rangle|
   \leq  \sum_{i\in J} \|B_i\| \cdot \|\phi_iv\|_1 \cdot \|\psi_i w\|_\infty \\
   &\leq & M \cdot \big(\sum_{x\in X} \sum_{i\in J} \phi_i(x) \|v(x)\|_E \big) \cdot \|w\|_\infty
   \leq M \cdot \|v\|_1,
\end{eqnarray*}
which implies that $\| \sum_{i \in J} \psi_i B_i \phi_iv \|_1 \leq M~\|v\|_1$. Due to density of finitely supported vectors in $\lo$, the above estimates hold for all $v\in \lo$.

Finally we deal with Case (2). For any $v\in \linf$ and any $x\in X$, we have
$$\big(\sum_{i \in J} \phi_i B_i \psi_iv\big)(x) = \sum_{i \in J}\phi_i(x)\cdot\big(B_i\psi_i v\big)(x)$$
which is a finite sum since $\phi_i(x)$ is non-zero for only finitely many $i \in J$. Furthermore, (\ref{EQ1}) still holds for such $v$, hence the result holds.
\end{proof}

\begin{lem}\label{cstrct op right}
Let $p \in \Dom$, $A$ be a band operator on $\lp$ with propagation at most $r$, $N=\sup_{x\in X} \sharp B(x,r)$ and $\epsilon>0$. Let $\{\phi_i\}_{i\in I}$ be a metric $p$-partition of unity on $X$, and $\{\psi_i\}_{i \in I}$ be a dual family of $\{\phi_i\}_{i \in I}$ when $p \in \{0,1,\infty\}$. Let $J \subseteq I$ and given a collection of bounded operators $\{B_i\}_{i \in J}$ on $\lp$ such that $M:=\sup_{i}\|B_i\|$ is finite.
\begin{enumerate}
  \item When $p<\infty$, consider the following:
    \begin{enumerate}
       \item $\sum_{i \in J} \phi_i^{p/q}B_i[\phi_i,A]$, if $p \in (1,\infty)$ and suppose $\{\phi_i\}_{i\in I}$ has $(r,\epsilon)$-variation;
       \item $\sum_{i \in J} \phi_i B_i [\psi_i,A]$, if $p=0$ and suppose $\{\psi_i\}_{i \in I}$ is $\epsilon/r$-Lipschitz;
       \item $\sum_{i \in J} \psi_i B_i [\phi_i,A]$, if $p=1$ and suppose $\{\phi_i\}_{i\in I}$ has $(r,\epsilon)$-variation.
    \end{enumerate}
    Each of them converges strongly to a band operator of norm $\leq\epsilon NM\|A\|$ on $\lp$.
  \item When $p=\infty$, suppose $\{\psi_i\}_{i \in I}$ is $\epsilon/r$-Lipschitz and consider $\sum_{i \in J} \phi_i B_i [\psi_i,A]$. It converges point-wise strongly to a band operator of norm $\leq\epsilon NM\|A\|$ on $\linf$.
\end{enumerate}
\end{lem}

\begin{proof}
By Lemma \ref{dec of BO}, $A$ has the form of $\sum_{k=1}^N f_k V_k$ where each $f_k:X \to \B(E)$ is a bounded function with norm at most $\|A\|$, and $V_k$ is a partial translation operator in $\BD$ defined by $t_k:D_k \to R_k$ with propagation at most $r$. For any function $\varphi: X \to \R$, $v \in \lp$ and $x\in X$, we have
\begin{equation}\label{EQ5}
\big( [\varphi,V_k]v \big)(x)=
\begin{cases}
  ~\big(\varphi(x)-\varphi(t_k^{-1}(x))\big)v(t_k^{-1}(x)), & x\in R_k; \\
  ~0, & \mbox{otherwise}.
\end{cases}
\end{equation}

Case (1)(a) is from \cite[Lemma 6.4]{vspakula2017metric}, hence omitted. Starting with Case (1)(b): for any finitely supported $v\in \lz$, $\sum_{i \in J} \phi_i B_i [\psi_i,A]v$ is a finite sum. For any $x\in X$,
\begin{eqnarray}\label{EQ2}
  \big\| \big(\sum_{i \in J} \phi_i B_i [\psi_i,A]v\big)(x) \big\|_E &=& \big\| \sum_{i \in J} \phi_i(x) \big(B_i [\psi_i,A]v\big)(x) \big\|_E
  \leq  \sum_{i \in J}\phi_i(x)\cdot\big\|B_i[\psi_i,A] v\big\|_\infty \nonumber \\
  &\leq & \sum_{i \in J}\phi_i(x)\cdot M \cdot \big(\sum_{k=1}^N \|A\| \cdot \|[\psi_i,V_k]v\|_\infty \big) \nonumber\\
  &\leq & M \cdot \big(\sum_{k=1}^N \|A\| \cdot \|[\psi_i,V_k]v\|_\infty \big).
\end{eqnarray}
Since each $\psi_i$ is $\epsilon/r$-Lipschitz and $d(x,t_k^{-1}(x)) \leq r$, we have $\big\| [\psi_i,V_k]v\big\|_\infty \leq \epsilon \|v\|_\infty$ by (\ref{EQ5}). Combining with (\ref{EQ2}), we have
\begin{equation}\label{EQ4}
\big\| \sum_{i \in J} \phi_i B_i [\psi_i,A]v\big\|_\infty \leq M\cdot \big( \sum_{k=1}^N \epsilon \cdot\|A\| \cdot \|v\|_\infty \big) \leq \epsilon MN~\|A\|~\|v\|_\infty.
\end{equation}
Due to density of finitely supported vectors in $\lz$, (\ref{EQ4}) holds for all $v\in \lz$.

We move on to Case (1)(c): Again, for any $v \in \lo$ with finite support, $\sum_{i \in J} \psi_i B_i [\phi_i,A]v$ is a finite sum. For any unit vector $w$ in $\ell^\infty_{E^*}(X) \cong \lo^*$, we have
\begin{eqnarray}\label{EQ3}
  \big| \big\langle \sum_{i \in J} \psi_i B_i [\phi_i,A]v, w \big\rangle \big| &\leq & \sum_{i\in J} \big|\big\langle [\phi_i,A]v, B_i^*\psi_i w\big\rangle\big|
   \leq  \sum_{i\in J} \|[\phi_i,A]v\|_1 \cdot \|B_i^*\| \cdot \|\psi_i w\|_\infty \nonumber\\
   &\leq & M \cdot \big(\sum_{i\in J} \|[\phi_i,A]v\|_1 \big)\cdot \|w\|_\infty
    \leq  M \cdot \sum_{i\in J} \big(\sum_{k=1}^N \|f_k[\phi_i,V_k]v\|_1\big) \nonumber\\
   & \leq & M \cdot\|A\|\cdot \sum_{k=1}^N \big(\sum_{i\in J} \|[\phi_i,V_k]v\|_1\big).
\end{eqnarray}
Note that $\{\phi_i\}_{i\in J}$ has $(r,\epsilon)$-variation and $d(x,t_k^{-1}(x)) \leq r$, hence by (\ref{EQ5}):
$$\sum_{i\in J}\|[\phi_i,V_k]v\|_1=\sum_{x\in R_k}\|v(t_k^{-1}(x))\|_E \cdot \big( \sum_{i \in J} |\phi_i(x)-\phi_i(t_k^{-1}(x))|\big) \leq \epsilon \|v\|_1$$
for each $k$. Combining with (\ref{EQ3}), we have
\begin{equation}\label{EQ6}
\big\| \sum_{i \in J} \psi_i B_i [\phi_i,A]v \big\| \leq M \cdot\|A\|\cdot \sum_{k=1}^N \epsilon \|v\|_1 \leq \epsilon MN~\|A\|~\|v\|_1.
\end{equation}
Due to density of finitely supported vectors in $\lo$, (\ref{EQ6}) holds for all $v\in \lo$.

Finally we deal with Case (2). For any $v\in \linf$ and any $x\in X$, we have
$$\big(\sum_{i \in J} \phi_i B_i [\psi_i,A]v\big)(x) = \sum_{i \in J}\phi_i(x)\cdot\big(B_i[\psi_i,A] v\big)(x)$$
which is a finite sum since $\phi_i(x)$ is non-zero for only finitely many $i \in J$. Furthermore, (\ref{EQ2}) and (\ref{EQ4}) still hold for such $v$, hence the result holds.
\end{proof}

On the other hand, we have the following lemma, whose proof is quite similar to the above, hence omitted.

\begin{lem}\label{cstrct op left}
Let $p \in \Dom$, $A$ be a band operator on $\lp$ with propagation at most $r$, $N=\sup_{x\in X} \sharp B(x,r)$ and $\epsilon>0$. For $p \in (1,\infty)$, let $\{\phi_i\}_{i\in I}$ be a metric $q$-partition of unity on $X$, where $q$ is the conjugate exponent of $p$; and for $p \in \{0,1,\infty\}$, let $\{\phi_i\}_{i\in I}$ be a metric $p$-partition of unity on $X$, and $\{\psi_i\}_{i \in I}$ be a dual family of $\{\phi_i\}_{i \in I}$. Let $J \subseteq I$ and given a collection of bounded operators $\{B_i\}_{i \in J}$ on $\lp$ such that $M:=\sup_{i}\|B_i\|$ is finite.
\begin{enumerate}
  \item When $p<\infty$, consider the following:
    \begin{enumerate}
       \item $\sum_{i \in J} [\phi_i,A] B_i \phi_i^{q/p}$, if $p \in (1,\infty)$ and suppose $\{\phi_i\}_{i\in I}$ has $(r,\epsilon)$-variation;
       \item $\sum_{i \in J} [\phi_i,A] B_i \psi_i$, if $p=0$ and suppose $\{\phi_i\}_{i\in I}$ has $(r,\epsilon)$-variation;
       \item $\sum_{i \in J} [\psi_i,A] B_i \phi_i$, if $p=1$ and suppose $\{\psi_i\}_{i \in I}$ is $\epsilon/r$-Lipschitz.
    \end{enumerate}
    Each of them converges strongly to a band operator of norm $\leq\epsilon NM\|A\|$ on $\lp$.
  \item When $p=\infty$, suppose $\{\phi_i\}_{i\in I}$ has $(r,\epsilon)$-variation and consider $\sum_{i \in J} [\phi_i,A] B_i \psi_i$. It converges point-wise strongly to a band operator of norm $\leq\epsilon NM\|A\|$ on $\linf$.
\end{enumerate}
\end{lem}

\subsection{Density of rich band operators}\label{density section}

Having established the technical lemmas above, we are now ready to prove Theorem \ref{density thm}. First let us fix a metric $1$-partition of unity $\{\phi_i\}_{i \in I}$ on $X$. For example, one may take an arbitrary disjoint bounded cover $\{U_i\}_{i \in I}$ of $X$, and $\phi_i$ to be the characteristic function of $U_i$. For each $n \in \N$, take $\{\psi_i^{(n)}\}_{i\in I}$ to be a $1/n$-Lipschitz dual family of $\{\phi_i\}_{i\in I}$.

\begin{lem}\label{app}
For $p \in \{0,1,\infty\}$ and $n\in \N$, define an operator $M_n: \Ap \to \Ap$ by
\begin{enumerate}
  \item $A \mapsto \sum_{i\in I}\phi_iA\psi_i^{(n)}$ if $p=0$ or $\infty$;
  \item $A \mapsto \sum_{i\in I}\psi_i^{(n)}A\phi_i$ if $p=1$.
\end{enumerate}
Then each $M_n$ is a well-defined linear operator of norm at most one. Moreover, $M_n(A)$ is a band operator and $M_n(A) \to A$ in norm as $n \to \infty$ for each $A \in \Ap$.
\end{lem}

\begin{proof}
By Lemma \ref{cstrct op}, $M_n$ is well-defined and has norm at most one. Clearly, $M_n(A)$ is a band operator for all $A\in \Ap$. For the convergence statement, we treat them separately.

$\bullet ~p=0$ or $\infty$: for each $n$ and any \emph{band} operator $A$,
$$M_n(A)=\sum_{i\in I}\phi_iA\psi_i^{(n)}=\sum_{i\in I} \phi_i\psi_i^{(n)}A + \sum_{i \in I} \phi_i[A,\psi_i^{(n)}]$$
where all sums converge strongly when $p=0$, and point-wise strongly when $p=\infty$ by Lemma \ref{cstrct op} and \ref{cstrct op right}. Since $\{\psi_i^{(n)}\}_{i \in I}$ is a dual family of $\{\phi_i\}_{i \in I}$, we have $\sum_{i\in I} \phi_i\psi_i^{(n)}A=\sum_{i\in I} \phi_iA$, which converges strongly to $A$ when $p=0$, and converges point-wise strongly to $A$ when $p=\infty$ by Lemma \ref{pwc lem1}. For the second term, it has norm at most $\|A\|N/n$ for some fixed $N$ by Lemma \ref{cstrct op right}, hence tends to $0$. Finally note that $\|M_n\| \leq 1$ for all $n$, so the result also holds for band-dominated operators.

$\bullet ~p=1$: for each $n$ and any \emph{band} operator $A$,
$$M_n(A)=\sum_{i\in I}\psi_i^{(n)}A\phi_i=\sum_{i\in I} A\psi_i^{(n)}\phi_i - \sum_{i \in I}[A,\psi_i^{(n)}] \phi_i$$
where all sums converge strongly by Lemma \ref{cstrct op} and \ref{cstrct op left}. Again we have $$\sum_{i\in I} A\psi_i^{(n)}\phi_i=\sum_{i\in I} A\phi_i=A.$$
For the second term, it has norm at most $\|A\|N/n$ for some fixed $N$ by Lemma \ref{cstrct op left}, hence tends to $0$. Finally note that $\|M_n\| \leq 1$ for all $n$, so the result holds for band-dominated operators as well.
\end{proof}

Now the proof of Theorem \ref{density thm} follows from the same argument as that of \cite[Theorem 6.6]{vspakula2017metric}, while replacing \cite[Corollary 6.5]{vspakula2017metric} with Lemma \ref{app} we just established above, hence omitted. Also notice that \cite[Corollary 6.5]{vspakula2017metric} is the only place where Property A is used in proving \cite[Theorem 6.6]{vspakula2017metric}, hence it is not necessary for Theorem \ref{density thm} since Lemma \ref{app} holds generally.

\section{The Main Theorem}\label{main thm section}

We are now in the position to state our main theorem, which characterises the properties of $\P$-Fredholmness and invertibility at infinity for a rich band-dominated operator in terms of its operator spectrum.

\begin{thm}\label{main thm}
Let $X$ be a space with Property A, $E$ a Banach space and $p \in \Dom$. Let $A$ be a rich band-dominated operator on $\lp$. Then the following are equivalent:
\begin{enumerate}
  \item $A$ is invertible at infinity;
  \item $A$ is $\P$-Fredholm;
  \item $A$ is invertible in $\Ap$ modulo $\Kp$;
  \item all the limit operators $\Phi_\omega(A)$ are invertible, and $\sup_{\omega \in \partial X}\|\Phi_\omega(A)^{-1}\|$ is finite;
  \item all the limit operators $\Phi_\omega(A)$ are invertible.
\end{enumerate}
\end{thm}

Note that for $p\in (1,\infty)$, the above theorem is exactly \cite[Theorem 5.1]{vspakula2017metric}. Our major work here is to fill the gaps of $p \in \{0,1,\infty\}$. Meanwhile, it is somewhat surprising that in these extreme cases, some parts of the theorem holds without the assumption of Property A, for example ``(4) $\Leftrightarrow$ (5)".

The proof occupies the rest of the paper, which is divided into several parts. In Section \ref{commt tech}, we deal with the equivalence between (1) and (2); from Section \ref{2 to 4 finite} to \ref{construct para}, we establish the equivalence between (2), (3) and (4); and finally, in Section \ref{unif bdd}, we prove ``(4) $\Leftrightarrow$ (5)".

Careful readers might have noticed that in the case of $p \in (1,\infty)$ \cite{vspakula2017metric}, ``(1) $\Leftrightarrow$ (2)" was proved implicitly via showing ``(1) $\Rightarrow$ (4) $\Rightarrow$ (3)". While in the case of $p=\infty$ as we will see in Section \ref{dual space}, it would be difficult to deduce (4) directly from (1). Hence, we provide a direct proof for ``(1) $\Leftrightarrow$ (2)" in Section \ref{commt tech}, which works for all $p \in \Dom$ as well.

\subsection{Commutant technique}\label{commt tech}

We prove ``(1) $\Leftrightarrow$ (2)" using a commutant technique, inspired by the classical limit operator theory for $\mathbb{Z}^n$ and part of the results in \cite{spakula2017relative}. Let us start with the following class of operators on $\lp$.
\begin{defn}\label{class C}
For $p \in \Dom$, we define the class $\mathcal{C}$ of bounded linear operators on $\lp$ as follows: $T \in \mathcal C$ \emph{if and only if} for any $\epsilon>0$, there exists some $L>0$ such that for any $L$-Lipschitz function $f \in C_b(X)_1$, we have $\|[T,f]\|< \epsilon$.
\end{defn}

We remark that elements in $\mathcal C$ are exactly quasi-local operators (see \cite{spakula2017relative}). The following lemma is direct from definition, hence the proof is omitted.
\begin{lem}\label{class C lem1}
For any $p \in \Dom$, $\mathcal C$ is a closed subalgebra in $\B(\lp)$ and $\BD \subseteq \mathcal C$. Hence $\Ap \subseteq \mathcal C$.
\end{lem}

\begin{rem}
Conversely, we would like to mention that very recently after the paper is finished, \v{S}pakula and the author \cite{SZ2018} showed that $\Ap=\mathcal C$ either if $X$ has Property A and $p \in (1,\infty)$, or without any assumption on $X$ when $p\in \{0,1,\infty\}$, partially using the tools developed in Section \ref{partition of unity} to approach the latter. While we do not need appeal to that result in this paper.
\end{rem}

\begin{lem}\label{class C lem2}
For any $p \in \Dom$, $\mathcal C$ is contained in $\Lp$.
\end{lem}

\begin{proof}
Fix $T \in \mathcal C$ and a finite subset $F \subseteq X$. For any $\epsilon>0$, there exists some $L>0$ such that $\|[T,f]\|<\epsilon$ for any $L$-Lipschitz function $f \in C_b(X)_1$. Now for any finite $G \subseteq X$ with $d(F,G^c)>1/L$, there exists an $L$-Lipschitz function $g \in C_b(X)$ with range in $[0,1]$ and satisfying $g|_F \equiv 1$ and $g|_{G^c}\equiv 0$. Hence $\|[T,g]\|< \epsilon$. Consequently,
$$\|P_FTQ_G\|=\|P_FgTQ_G\| \leq \|P_FTgQ_G\| + \|P_F[T,g]Q_G\| \leq \|[T,g]\| < \epsilon.$$
We finish the proof.
\end{proof}

\begin{proof}[Proof of Theorem \ref{main thm}, ``(1) $\Leftrightarrow$ (2)"]
Clearly, (2) implies (1). For the other direction, assume that there exists some $B \in \B(\lp)$ such that $AB=\Id+K_1$ and $BA=\Id+K_2$ for some $K_1,K_2 \in \Kp$. By Lemma \ref{class C lem1}, $A \in \Ap \subseteq \mathcal C$. We claim that $B \in \mathcal C$ as well, hence $B \in \Lp$ by Lemma \ref{class C lem2}.

Indeed for any $\epsilon>0$, there exists some $L_1>0$ such that $\|[A,f]\|<\epsilon/(3\|B\|^2)$ for any $L_1$-Lipschitz function $f \in C_b(X)_1$. On the other hand, since $K_1,K_2 \in \Kp$, there exists some finite subset $F_0 \subseteq X$ such that
\begin{equation}\label{EQ18}
\|Q_{F_0}K_1\| < \frac{\epsilon}{12\|B\|} \andx \|K_2Q_{F_0}\| < \frac{\epsilon}{12\|B\|}.
\end{equation}
Choose a point $x_0 \in X$, and take $L_2=\epsilon/[6\mathrm{diam}F_0\cdot\|B\|\cdot(\|K_1\|+\|K_2\|)]$, $L:=\min\{L_1,L_2\}$.

For any $L$-Lipschitz function $f \in C_b(X)_1$, we take $\tilde{f}=f-f(x_0)$.  Then
\begin{equation}\label{EQ11}
[f,B]=[\tilde{f},B]=(BA-K_2)\tilde{f}B-B\tilde{f}(AB-K_1)=B[A,f]B+B\tilde{f}K_1-K_2\tilde{f}B.
\end{equation}
For the above $F_0 \subseteq X$, we have
\begin{equation}\label{EQ12}
B\tilde{f}K_1=B\tilde{f}Q_{F_0}K_1+BP_{F_0}\tilde{f}K_1, \andx K_2\tilde{f}B=K_2Q_{F_0}\tilde{f}B+K_2P_{F_0}\tilde{f}B.
\end{equation}
By (\ref{EQ18}), we obtain
\begin{equation}\label{EQ13}
\|B\tilde{f}Q_{F_0}K_1\| < \epsilon/6 \andx \|K_2Q_{F_0}\tilde{f}B\| < \epsilon/6.
\end{equation}
Furthermore, since $x_0 \in F_0$, we have
\begin{equation*}
\|P_{F_0}\tilde{f}\|\leq L_2\cdot\mathrm{diam}F_0=\frac{\epsilon}{6\|B\|\cdot(\|K_1\|+\|K_2\|)},
\end{equation*}
which implies that
\begin{equation}\label{EQ14}
\|BP_{F_0}\tilde{f}K_1\|\leq \epsilon/6, \andx \|K_2P_{F_0}\tilde{f}B\| \leq \epsilon/6.
\end{equation}
Hence combining (\ref{EQ11}), (\ref{EQ12}), (\ref{EQ13}) and (\ref{EQ14}), we have:
\begin{eqnarray*}
  \|[f,B]\| &\leq & \|B[A,f]B\| + \|B\tilde{f}K_1\| + \|K_2\tilde{f}B\| \\
   &\leq & \|B\|\cdot \frac{\epsilon}{3\|B\|^2} \cdot \|B\| + \|B\tilde{f}Q_{F_0}K_1\| + \|BP_{F_0}\tilde{f}K_1\| + \|K_2Q_{F_0}\tilde{f}B\| + \|K_2P_{F_0}\tilde{f}B\| \\
   & \leq & \frac{\epsilon}{3} + \frac{\epsilon}{6} + \frac{\epsilon}{6} + \frac{\epsilon}{6} + \frac{\epsilon}{6} = \epsilon.
\end{eqnarray*}
Hence we finish the proof.
\end{proof}

From the above proof, we obtain a bonus result as follows.
\begin{cor}\label{closed under inv}
Let $p \in \Dom$, $A,B \in \B(\lp)$ and $A \in \mathcal C$. If $AB-\Id$ and $BA-\Id$ belong to $\Kp$, then $B \in \mathcal{C}$ as well. In particular, $\mathcal{C}$ is closed under taking inverses.
\end{cor}


\subsection{``(2) $\Rightarrow$ (4)" for finite $p$}\label{2 to 4 finite}
We move on to prove ``(2) $\Rightarrow$ (4)" without the assumption of Property A. First, in this subsection, we deal with the case of finite $p$. The idea follows partially from \cite{vspakula2017metric} together with the commutant technique developed above. However, things become complicated when $p=\infty$ due to lack of characterisation of the dual space of $\linf$. Hence we leave it to the next subsection after new tools are introduced.

Let us start with the following lemma. The proof is almost the same as that of \cite[Lemma 5.3]{vspakula2017metric}, together with Theorem \ref{density thm} and Proposition \ref{limit op homo}, hence omitted.

\begin{lem}\label{norm app}
For $p \in \Dom$, let $A$ be a band-dominated operator on $\lp$ rich at $\omega \in \partial X$. For any finitely supported unit vector $v \in \ell^p_E(X(\omega))$, finite subset $G \subseteq X$ and $\epsilon>0$, there exists a unit vector $w \in \lp$ such that
$$\big| \|Aw\|-\|\Phi_\omega(A)v\| \big| < \epsilon$$
and $\supp(w) \cap G = \emptyset$.
\end{lem}

\begin{prop}\label{bdd below}
Let $p \in \Dom$, and $A$ be a rich band-dominated operator on $\lp$ which is invertible at infinity. Then the operator spectrum $\sigma_{op}(A)$ is \emph{uniformly bounded below}, i.e., there exists some $M>0$ such that for any $\omega \in \partial X$ and $v \in \ell^p_E(X(\omega))$,
$$\|\Phi_\omega(A)v\| \geq M\|v\|. $$
\end{prop}

\begin{proof}
Since $A$ is invertible at infinity, there exists a bounded operator $B$ on $\lp$ such that $K_1:=AB-\Id$ and $K_2:=BA-\Id$ are in $\Kp$. We claim that for any $\omega \in \partial X$ and finitely supported $v\in \ell^p_E(X(\omega))$,
\begin{equation}\label{EQ16}
\|\Phi_\omega(A)v\| \geq \frac{\|v\|}{\|B\|}.
\end{equation}

To prove the claim, we fix a $\omega \in \partial X$, a finitely supported $v\in \ell^p_E(X(\omega))$ and an $\epsilon>0$. Since $K_2\in \Kp$, there exists some finite $G\subseteq X$ such that $\|K_2Q_G\|<\epsilon$. By Lemma \ref{norm app}, there exists a unit vector $w \in \lp$ such that
$$\big| \|Aw\|-\|\Phi_\omega(A)v\| \big| < \epsilon$$
and $\supp(w) \cap G = \emptyset$. Hence we have $\|K_2w\|< \epsilon$, and
$$\|B\|\cdot \|Aw\| \geq \|BAw\| = \|(1-K_2)w\| \geq \|w\|-\|K_2w\|\geq 1-\epsilon.$$
Combining them together, we have
$$\|\Phi_\omega(A)v\| \geq \|Aw\|-\epsilon \geq \frac{1-\epsilon}{\|B\|} -\epsilon.$$
Letting $\epsilon \to 0$, we obtain (\ref{EQ16}) as required and the claim holds.

Now for $p<\infty$, (\ref{EQ16}) holds for all vectors in $\ell^p_E(X(\omega))$ since finitely supported vectors are dense in $\ell^p_E(X(\omega))$.

For $p=\infty$, we fix a vector $v \in \ell^\infty_E(X(\omega))$ and an $\epsilon>0$. There exists some $\alpha \in X(\omega)$ such that $\|v(\alpha)\|_E>\|v\|_\infty-\epsilon$. Applying Proposition \ref{limit op homo} and Lemma \ref{class C lem1}, $\Phi_\omega(A) \in \mathcal{A}^\infty_E(X(\omega)) \subseteq \mathcal C$. Hence there exists some $L>0$ such that for any $L$-Lipschitz function $f \in C_b(X(\omega))_1$, we have $\|[\Phi_\omega(A),f]\|<\epsilon$. Take an $L$-Lipschitz function $f$ with range in $[0,1]$ and supported in the $1/L$-neighbourhood of $\alpha$ with $f(\alpha)=1$. Note that $fv$ is finitely supported, hence (\ref{EQ16}) holds for $fv$. Furthermore, $$\|fv\|_\infty \geq \|f(\alpha)v(\alpha)\|_E \geq \|v\|_\infty-\epsilon.$$
Combining them together, we have
\begin{eqnarray*}
\|\Phi_\omega(A)v\|_\infty &\geq& \|f\Phi_\omega(A)v\|_\infty \geq \|\Phi_\omega(A)fv\|_\infty - \|[\Phi_\omega(A),f]v\|_\infty \\[0.2cm]
& \geq & \frac{\|fv\|_\infty}{\|B\|} - \epsilon\|v\|_\infty \geq \frac{\|v\|_\infty-\epsilon}{\|B\|} - \epsilon\|v\|_\infty.
\end{eqnarray*}
Taking $\epsilon \to 0$, (\ref{EQ16}) holds for any $v\in \ell^\infty_E(X(\omega))$ as required. We finish the proof.
\end{proof}

We need the following auxiliary lemma concerning adjoints of limit operators. The proof is straightforward, hence omitted. Notice that for $p=0$, we set $q=1$ as its \emph{conjugate exponent}; and for $p\in [1,\infty)$, we set its \emph{conjugate exponent} to be $q \in (1,\infty]$ satisfying $1/p+1/q=1$.

\begin{lem}\label{adj limit op}
Let $p\in \dom$ and $q$ be the conjugate exponent. Let $A$ be a band-dominated operator on $\lp$, rich at $\omega$. Then $A^*$ is a band-dominated operator on $\ell_{E^*}^q(X)$, rich at $\omega$. Furthermore, we have
$$\Phi_\omega(A)^*=\Phi_\omega(A^*),$$
which implies that if $A$ is rich, then
$$\sigma_{op}(A^*)=\{B^*: B \in \sigma_{op}(A)\}.$$
\end{lem}

\begin{proof}[Proof of Theorem \ref{main thm}, ``(2) $\Rightarrow$ (4), $p<\infty$"]
Let $A$ be a $\P$-Fredholm rich operator in $\Ap$ for $p<\infty$, then trivially $A$ is invertible at infinity. By Proposition \ref{bdd below}, the operator spectrum $\sigma_{op}(A)$ is uniformly bounded below.

On the other hand, since $p\in \dom$, we take $q \in [1,\infty]$ to be its conjugate exponent. Consider the adjoint $A^*$, which is a rich band-dominated operator on $\ell_{E^*}^q(X)$ by Lemma \ref{adj limit op}. Since $A$ is invertible at infinity, there exists a bounded operator $B$ on $\lp$ such that $K_1:=AB-\Id$ and $K_2:=BA-\Id$ are in $\Kp$, which implies that $K_1^*=B^*A^*-\Id$ and $K_2^*=A^*B^*-\Id$ are in $\mathcal{K}^q_{E^*}(X)$. Hence, $A^*$ is invertible at infinity as well. Applying Proposition $\ref{bdd below}$ to $A^*$, the operator spectrum $\sigma_{op}(A^*)$ is uniformly bounded below, which implies $\{\Phi_\omega(A)^*\}_{\omega \in \partial X}$ is uniformly bounded below by Lemma \ref{adj limit op}. Therefore, condition (4) holds and we finish the proof.
\end{proof}

\begin{rem}
The above proof does not work for $p=\infty$, since the dual of $\linf$ is no longer an $\ell^p$-type space. Hence we cannot refer to Proposition \ref{bdd below} any more. New techniques are required, which are introduced in the next subsection.
\end{rem}

\subsection{Dual space arguments for $p=\infty$}\label{dual space}
Now we focus on the case of $p=\infty$ and finish the proof of ``(2) $\Rightarrow$ (4)" completely. The key ingredient here is the dual space argument, which showed its power in the classical limit operator theory. The idea is that properties of operators on $\ell^\infty$-spaces can be characterised via those of its ``\emph{double predual}" on $c_0$-spaces.

Recall that from Lemma \ref{Y0}, $\lz=\linfz$ is the closure of $\bigcup_{F \in \F}P_F(\linf)$ in $\linf$. For an operator $A \in \Linf$, $\lz$ is $A$-invariant by Lemma \ref{inv}. Denote its restriction $A|_{\lz}$ by $A_0 \in \B(\lz)$. We have the following elementary observation.
\begin{lem}\label{restriction}
For $A \in \Linf, \Ainf$ or $\Kinf$, its restriction $A_0 \in \Lz, \Az$ or $\Kz$ respectively, and $\|A\|=\|A_0\|$. Furthermore, if $A \in \Ainf$ is rich at $\omega\in \partial X$, so is $A_0$. And we have
$$\Phi_\omega(A)\big|_{\ell^0_E(X(\omega))} = \Phi_\omega(A_0).$$
\end{lem}

The following proposition is taken from \cite{chandler2011limit}. Although the setting there is $X=\mathbb{Z}^N$, the proof applies to any general space $X$, hence omitted.
\begin{prop}[Corollary 6.20, \cite{chandler2011limit}]\label{restriction inv}
For $A \in \Linf$, it holds that $A$ is invertible \emph{if and only if} $A_0$ is invertible. In this case, we have $A_0^{-1}=(A^{-1})|_{\lz}$.
\end{prop}


Consequently, we have the following result for operator spectra under restriction.
\begin{cor}\label{res limit}
Let $A \in \Ainf$ be a rich band-dominated operator, and $A_0\in \Az$ be its restriction on $\lz$. Then:
\begin{equation}\label{EQ17}
\sigma_{op}(A_0)=\big\{\Phi_\omega(A)\big|_{\ell^0_E(X(\omega))}: \omega \in \partial X \big\}.
\end{equation}
In particular, the invertibility of all limit operators of $A_0$ with uniform boundedness of their inverses is equivalent to the same property for the limit operators of $A$.
\end{cor}

\begin{proof}
The proof of (\ref{EQ17}) follows from Lemma \ref{restriction}. For the second statement, we know that $\Phi_\omega(A) \in \mathcal{A}_E^\infty(X(\omega))$ for all $\omega \in \partial X$ by Proposition \ref{limit op homo}. Hence by Proposition \ref{restriction inv}, $\Phi_\omega(A)$ is invertible if and only if $\Phi_\omega(A)|_{\ell^0_E(X(\omega))}$ is, and we have
$$(\Phi_\omega(A)|_{\ell^0_E(X(\omega))})^{-1} = \Phi_\omega(A)^{-1}|_{\ell^0_E(X(\omega))}.$$
By Corollary \ref{closed under inv}, $\Phi_\omega(A)^{-1} \in \mathcal{L}_E^\infty(X(\omega))$. Hence by Lemma \ref{restriction} again, we have
$$\big\| (\Phi_\omega(A)|_{\ell^0_E(X(\omega))})^{-1} \big\| = \big\|\Phi_\omega(A)^{-1}|_{\ell^0_E(X(\omega))}\big\|=\|\Phi_\omega(A)^{-1}\|.$$
So we finish the proof.
\end{proof}

\begin{proof}[Proof of Theorem \ref{main thm}, ``(2) $\Rightarrow$ (4), $p = \infty$"]
Since $A$ is $\P$-Fredholm, there exists some $B \in \Linf$ such that $K_1:=AB-\Id$ and $K_2:=BA-\Id$ are in $\Kinf$. By Lemma \ref{inv}, $\lz$ is invariant under $B$. Denote its restriction on $\lz$ by $B_0$. By Lemma \ref{restriction}, $K_1|_{\lz}$ and $K_2|_{\lz}$ are in $\Kz$. Hence $A_0$ is $\P$-Fredholm as well. Since we already proved ``(2) $\Rightarrow$ (4)" for $p=0$ in Section \ref{2 to 4 finite}, all limit operators of $A_0$ are invertible and their inverses are uniformly bounded. Finally, by Corollary \ref{res limit}, the same holds for $A$. So we finish the proof.
\end{proof}

\subsection{Constructing parametrices}\label{construct para}
Now we move on to prove ``(4) $\Rightarrow$ (3)", following the ideas in \cite{lindner2006infinite} and \cite{vspakula2017metric} and using the tools we developed in Section \ref{partition of unity} instead. First we recall the following result to construct parametrices.

\begin{lem}\label{local parametrices}
Let $p \in \Dom$, and $A$ be a rich band operator on $\lp$ such that all limit operators $\Phi_\omega(A)$ are invertible, and $M:=\sup_{\omega}\|\Phi_\omega(A)^{-1}\|$ is finite. Let $\{V_i\}_{i \in I}$ be a uniformly bounded cover of $X$ with finite multiplicity. Then there exists a finite subset $K$ in $I$ such that for all $i \in I \setminus K$, there exist operators $B_i, C_i$ on $\lp$ with norm at most $M$ and satisfying
\begin{equation}\label{EQ15}
B_iAP_{V_i}=P_{V_i}=P_{V_i}AC_i.
\end{equation}
\end{lem}

The proof is the same as that of \cite[Lemma 6.8]{vspakula2017metric}, hence omitted. Now we prove ``(4) $\Rightarrow$ (3)" first for rich \emph{band} operators.

\begin{prop}\label{band para}
Let $p \in \Dom$ and $A$ be a rich band operator. Assume that all the limit operators $\Phi_\omega(A)$ are invertible, and $M:=\sup_{\omega}\|\Phi_\omega(A)^{-1}\|$ is finite. Assume the space $X$ has Property A, then there exists an operator $B$ in $\Ap$ which is an inverse for $A$ modulo $\Kp$ and satisfying $\|B\| \leq 2M$.
\end{prop}

\begin{proof}
The case of $p \in (1, \infty)$ follows from \cite[Proposition 6.7]{vspakula2017metric}, hence omitted. We focus on $p\in \{0,1,\infty\}$.

Let $\epsilon=1/(2MN\|A\|)$, where $N=\sup_{x\in X} \sharp B(x,\ppg(A))$. Let $\{\phi_i\}_{i\in I}$ be a metric $1$-partition of unity with $(\ppg(A),\epsilon)$-variation, and $\{\psi_i\}_{i \in I}$ be an $\epsilon/\ppg(A)$-Lipschitz dual family for $\{\phi_i\}_{i \in I}$. Applying Lemma \ref{local parametrices} to the cover $\{\supp(\psi_i)\}_{i \in I}$ of $X$, we obtain a finite subset $K$ in $I$, and operators $B_i,C_i$ satisfying (\ref{EQ15}) for $i \in I \setminus K$. Set $P_i:=P_{\supp(\psi_i)}$. Now we divide the proof into two cases.

$\bullet ~p=0$ or $\infty$: by Lemma \ref{cstrct op}, the sum $\sum_{i \in I\setminus K} \phi_iB_i\psi_i$ converges strongly or point-wise strongly to an operator on $\lp$ with norm at most $M$. Consider
\begin{eqnarray*}
  \big( \sum_{i \in I\setminus K} \phi_iB_i\psi_i \big)A &=& \sum_{i \in I\setminus K} \phi_iB_i\psi_iA =  \sum_{i \in I\setminus K} \phi_iB_iAP_i\psi_i + \underbrace{\sum_{i \in I\setminus K} \phi_iB_i[\psi_i,A]}_{=:T_0} \\[0.2cm]
   &=& \sum_{i \in I\setminus K} \phi_i \psi_i + T_0 = (\Id + T_0) - \sum_{i \in K} \phi_i,
\end{eqnarray*}
where the first equality follows from Lemma \ref{pwc lem1} when $p=\infty$. All the series above converge strongly or point-wise strongly and by Lemma \ref{cstrct op right}, $\|T_0\| \leq 1/2$. Hence $\Id +T_0$ is invertible and $\|(\Id +T_0)^{-1}\| \leq 2$. Furthermore, $(\Id +T_0)^{-1}$ is given by a Neumann series and thus still in $\Ap$. Hence the operator
$$A_L:=(1+T_0)^{-1} \big( \sum_{i \in I\setminus K} \phi_iB_i\psi_i \big)$$
is in $\Ap$ with norm at most $2M$, satisfying $A_L\cdot A-\Id \in \Kp$. For the right inverse, set $T'_0:=\sum_{i \in I\setminus K} [A,\phi_i]C_i\psi_i$ and
$$A_R:=\big( \sum_{i \in I\setminus K} \phi_iC_i\psi_i \big)(1+T_0')^{-1}.$$
Similarly, we have that $A\cdot A_R - \Id \in \Kp$.

$\bullet ~p=1$: Setting
$$T_1:=\sum_{i \in I\setminus K} \psi_iB_i[\phi_i,A], ~~ T_1':=\sum_{i \in I\setminus K} [A,\psi_i]C_i\phi_i,$$
and
$$A_L:=(1+T_1)^{-1} \big( \sum_{i \in I\setminus K} \psi_iB_i\phi_i \big), ~~ A_R:=\big( \sum_{i \in I\setminus K} \psi_iC_i\phi_i \big)(1+T_1')^{-1}.$$
Similarly, by (1)(c) of Lemma \ref{cstrct op}, \ref{cstrct op right} and \ref{cstrct op left}, $A_L,A_R$ fulfill the requirements.
\end{proof}

Now the proof of Theorem \ref{main thm}, ``(4) $\Rightarrow$ (3)" is the same as that of \cite[Theorem 5.1, ``(2) $\Rightarrow$ (1)"]{vspakula2017metric}, replacing \cite[Proposition 6.7]{vspakula2017metric} with Proposition \ref{band para} above, hence omitted. On the other hand, notice that ``(3) $\Rightarrow$ (2)" holds trivially, hence combining with Section \ref{2 to 4 finite} and \ref{dual space}, we obtain the equivalence between (2), (3) and (4) for all $p \in \Dom$.

\subsection{Uniform boundedness}\label{unif bdd}

Finally we deal with the equivalence between (4) and (5). The case of $p \in (1,\infty)$ was proved in \cite[Theorem 7.3]{vspakula2017metric} under the assumption of Property A. Here we focus on $p \in \{0,1,\infty\}$ and show that Property A is not necessary for ``(4) $\Leftrightarrow$ (5)" in this case. Let us start with the following notion.

\begin{defn}[\cite{vspakula2017metric}]
Let $E',E"$ be Banach spaces, and $T:E' \to E"$ be a bounded linear operator. Define the \emph{lower norm} of $T$ to be
$$\nu(T):=\inf\{\|Tv\|_{E"}: v\in E',~~\|v\|_{E'}=1\}.$$
If $E'=\ell^p_E(X)$, $E"=\ell^p_E(Y)$ for spaces $X,Y$ and $s \geq 0$, we define
$$\nu_s(T):=\inf\{\|Tv\|_{E"}: v\in E',~~\|v\|_{E'}=1,~~\mathrm{diam}(\supp(v)) \leq s\}.$$
Furthermore, if $F \subseteq X$ and $A\in \B(\lp)$, denote the restriction of $A$ to $\ell^p_E(F)$ by $A|_F: \ell^p_E(F) \to \lp$.
\end{defn}

\begin{rem}
If $A$ is invertible, then $\nu(A)=1/\|A^{-1}\|$. Also, we have the estimate that $|\nu(A)-\nu(B)| \leq \|A-B\|$ for any operators $A,B$.
\end{rem}

First we deal with the case of $p=\infty$, which is clearly a corollary of the following theorem.
\begin{thm}\label{bdd below thm}
Let $X$ be a space, and $A$ be a rich band-dominated operator on $\linf$. Then there exists an operator $C \in \sigma_{op}(A)$ with $\nu(C)=\inf\{\nu(B):B \in \sigma_{op}(A)\}$.
\end{thm}

The proof is divided into several pieces. First recall that the class $\mathcal C$ of operators on $\linf$ is defined in Section \ref{commt tech}, such that all band operators sit inside $\mathcal C$. Furthermore, we have the following \emph{uniform} version.

\begin{lem}\label{unif commt}
Let $X$ be a space. For any $M,r \geq 0$ and $\epsilon>0$, there exists some $L\geq0$ such that for any $L$-Lipschitz function $f \in C_b(X)_1$ and any band operator $A$ on $\linf$ with propagation at most $r$ and norm at most $M$, we have $\|[A,f]\|<\epsilon$.
\end{lem}

\begin{proof}
Let $N=\sup_{x\in X} \sharp B(x,r)$. By Lemma \ref{dec of BO}, there exist multiplication operators $f_1,\ldots,f_N$ with $\|f_k\| \leq \|A\| \leq M$, and partial translation operators $V_1,\ldots,V_N$ of propagation at most $r$ defined by partial translations $t_k: D_k \to R_k$ such that $A=\sum_{k=1}^N f_kV_k$. Note that for any $k=1,\ldots,N$, $f \in C_b(X)$ and $v \in \linf$, we have
\begin{equation*}
\big([f_kV_k,f]v\big)(x)=
\begin{cases}
  ~\big(f(t_k^{-1}x)-f(x)\big)\cdot f_k(x)\big( v(t_k^{-1}x) \big), & x\in R_k; \\
  ~0, & \mbox{otherwise.}
\end{cases}
\end{equation*}
Hence, taking $L=\epsilon/(rMN)$ and for any $L$-Lipschitz function $f \in C_b(X)_1$, we have
$$\|[A,f]\| \leq \sum_{k=1}^N \|[f_kV_k,f]\| \leq NLrM =\epsilon.$$
So we finish the proof.
\end{proof}

The following result asserts the phenomenon of ``\emph{lower norm localisation}", which should be regarded as a dual version of \emph{operator norm localisation} introduced in \cite{chen2008metric}.
\begin{prop}\label{local norm}
Let $X$ be a space. For any $\delta>0$, $M\geq 0$ and $r \geq 0$, there exists $s=s(\delta, M, r) \geq 0$ such that
$$\nu(A|_F) \leq \nu_s(A|_F) \leq \nu(A|_F) + \delta$$
for any $A \in \B(\linf)$ with propagation at most $r$, norm at most $M$, and any $F \subseteq X$.
\end{prop}

\begin{proof}
The first inequality is trivial. For the second, fix $\delta > 0$, $M,r \geq 0$ and $F \subseteq X$. By Lemma \ref{unif commt}, there exists $L=L(\delta, r, M)>0$ satisfying the requirement thereby for $\epsilon=\delta/4, M$ and $r$. Set $s=2/L$ and take a unit vector $v \in \ell^\infty_E(F)$ such that $\|Av\| \leq \nu(A|_F)+ \delta/4$, and $x_0 \in F$ such that $\|v(x_0)\|>1-\delta/(2\nu_s(A|_F))$. Also take an $L$-Lipschitz function $f\in C_b(X)_1$, supported in $B(x_0,1/L)$ and $f(x_0)=1$. Hence,
$$\|fv\| \geq \|f(x_0)v(x_0)\| > 1-\frac{\delta}{2\nu_s(A|_F)},$$
and $fv\in \ell^\infty_E(F)$ has support in $B(x_0,1/L) \cap F$. Furthermore, we have
$$\|Av\| \geq \|fAv\| \geq \|Afv\|-\|[A,f]v\| \geq \nu_s(A|_F)\|fv\|-\frac{\delta}{4} \geq \nu_s(A|_F) - \frac{\delta}{2}-\frac{\delta}{4}.$$
Combine them together, we have
$$\nu_s(A|_F) \leq \|Av\|+\frac{3\delta}{4} \leq \nu(A|_F) + \frac{\delta}{4}+\frac{3\delta}{4} = \nu(A|_F) + \delta.$$
So we finish the proof.
\end{proof}

Consequently, we have the following corollary by the same proof as \cite[Corollary 7.10]{vspakula2017metric}, replacing \cite[Proposition 7.6]{vspakula2017metric} with Proposition \ref{local norm} above, hence omitted.

\begin{cor}\label{nu local}
Let $X$ be a space and $A$ be a band-dominated operator on $\linf$. Then for any $\delta>0$, there exists $s>0$ such that
$$\nu(\Phi_\omega(A)|_F) \leq \nu_s(\Phi_\omega(A)|_F) \leq \nu(\Phi_\omega(A)|_F) + \delta$$
for all $\omega \in \partial X$ such that $A$ is rich at $\omega$, and all $F \subseteq X(\omega)$.
\end{cor}

Using Corollary \ref{nu local} instead of \cite[Corollary 7.10]{vspakula2017metric}, now the rest of the proof for Theorem \ref{bdd below thm} follows exactly the same as that of \cite[Theorem 7.3]{vspakula2017metric}, occupying \cite[Section 7.2]{vspakula2017metric}. Hence we omit the rest of the proof. Also notice that \cite[Corollary 7.10]{vspakula2017metric} is the only place where Property A is used to prove \cite[Theorem 7.3]{vspakula2017metric}, so it is unnecessary for Theorem \ref{bdd below thm}.

Consequently, we obtain Theorem \ref{main thm}, ``(5) $\Leftrightarrow$ (4), $p=\infty$". Next we deal with the remained case of $p \in \{0,1\}$, and we need the following extension result:
\begin{lem}[Lemma 3.18, \cite{chandler2011limit}]\label{ext lem}
Every $A \in \Sz$ has a unique extension to an operator $\hat{A} \in \Sinf$. It holds that $\|\hat{A}\|=\|A\|$, and if $A \in \Lz$, $\Kz$ or $\Az$, then $\hat{A} \in \Linf$, $\Kinf$ or $\Ainf$, respectively.
\end{lem}

\begin{proof}[Proof of Theorem \ref{main thm}, ``(5) $\Leftrightarrow$ (4)"]
The case of $p \in (1,\infty)$ is from \cite[Theorem 5.1, ``(3) $\Rightarrow$ (2)"]{vspakula2017metric}, and the case of $p=\infty$ is a direct corollary of Theorem \ref{bdd below thm}.

For $p=0$, let $A\in \Az$ be a rich operator satisfying condition (5). By Lemma \ref{ext lem} and Proposition \ref{restriction inv}, we may consider its extension $\hat{A} \in \Ainf$ and condition (5) still holds for $\hat{A}$. Since we already proved the result for $p=\infty$, condition (4) holds for $\hat{A}$. By Corollary \ref{res limit}, we know that (4) also holds for $A$ as well.

Finally, we deal with the case of $p=1$. Assume $A$ is a rich band-dominated operator on $\lo$, and that all $B \in \sigma_{op}(A)$ are invertible. Hence all their adjoints are invertible as well, which compose $\sigma_{op}(A^*)$ by Lemma \ref{adj limit op}. From the result for $p=\infty$ proved above, we have that
$$\sup_{B \in \sigma_{op}(A)}\|B^{-1}\| = \sup_{B \in \sigma_{op}(A)}\|(B^*)^{-1}\| = \sup_{C=B^* \in \sigma_{op}(A^*)} \|C^{-1}\| <\infty.$$
So we finish the proof.
\end{proof}

\bibliographystyle{abbrv}
\bibliography{bibLIMIT}

\end{document}